\documentclass[letterpaper,12pt]{article}
\usepackage{amsmath,amssymb,amsfonts,epsf,epsfig}
\newtheorem{theo}{Theorem}[section]
\newtheorem{prop}[theo]{Proposition}
\newtheorem{prob}[theo]{Problem}
\newtheorem{lemma}[theo]{Lemma}

\newtheorem{cor}[theo]{Corollary}

\newtheorem{conj}[theo]{Conjecture}
\newtheorem{cons}[theo]{Construction}

\newenvironment{proof}{\noindent {\sc Proof}.}
                {\phantom{a} \hfill \framebox[2.2mm]{ } \bigskip}

\newenvironment{proofof}{\noindent {\sc Proof of Theorem}}
                {\phantom{a} \hfill \framebox[2.2mm]{ } \bigskip}

\newcommand{\ZZ}{\mathbb{Z}}

\newcommand{\C}{{\rm Circ}}

\title{On the directed Oberwolfach Problem \\ with equal cycle lengths: the odd case}
\author{Andrea Burgess \\
{\small University of New Brunswick} \\ \\
Nevena Franceti\'{c} \\
{\small University of Ottawa} \\ \\
Mateja \v{S}ajna\footnote{Corresponding author. Email: msajna$@$uottawa.ca. Mailing address: Department of Mathematics and Statistics, University of Ottawa, 585 King Edward Avenue, Ottawa, ON, K1N 6N5, Canada.} \\
{\small University of Ottawa}}

\begin{document}
\maketitle \baselineskip 17pt

\begin{abstract}
We show that the complete symmetric digraph $K_{2m}^\ast$ admits a resolvable decomposition into directed cycles of length $m$ for all odd $m$, $5 \le m \le 49$. Consequently, $K_{n}^\ast$ admits a resolvable decomposition into directed cycles of length $m$ for all  $n \equiv 0 \pmod{2m}$ and odd $m$, $5 \le m \le 49$.

\end{abstract}

%%%%%%%%%%%%%%%%%
% INTRODUCTION %
%%%%%%%%%%%%%%%%%

\section{Introduction}

In this paper, we are concerned with the problem of decomposing the complete symmetric digraph $K_n^\ast$ into spanning subdigraphs, each a vertex-disjoint union of directed cycles of length $m$. Thus, we are interested in the following problem.

\begin{prob}\label{prob}
Determine the necessary and sufficient conditions on $m$ and $n$ for the complete symmetric digraph $K_n^\ast$ to admit a resolvable decomposition into directed $m$-cycles.
\end{prob}

In the design-theoretic literature, such decompositions have also been called Mendelsohn designs \cite{ColDin}. Problem~\ref{prob} can also be viewed as the directed version of the well-known Oberwolfach Problem with uniform cycle lengths, which was completely solved in \cite{AlsHag,AlsSch}.

It is easily seen that $K_n^\ast$ admits a resolvable decomposition into directed $m$-cycles only if $m$ divides $n$, and this condition is obviously sufficient if $m=2$. Problem~\ref{prob} has also been solved previously for $m=3$  \cite{BerGerSot} and for $m=4$ \cite{BenZha,AdaBry}: the necessary conditions are sufficient except for $(m,n)=(3,6)$ and $(4,4)$. More recently, two of the present authors showed the following.

\begin{theo} \cite{BurSaj}\label{the:BurSaj}
Let $m$ and $n$ be integers with $5 \le m \le n$. Then the following hold.
\begin{enumerate}
\item Let $m$ be even, or $m$ and $n$ be both odd. Then there exists a resolvable decomposition of $K_n^\ast$ into  directed $m$-cycles if and only if $m$ divides $n$ and $(m,n)\ne(6,6)$.
\item If there exists a resolvable decomposition of $K_{2m}^\ast$ into directed $m$-cycles, then there exists a resolvable decomposition of $K_n^\ast$ into  directed $m$-cycles whenever $n \equiv 0 \pmod{2m}$.
\end{enumerate}
\end{theo}

In the same paper, we also posed the following conjecture.

\begin{conj}\cite{BurSaj}\label{conj}
Let $m$ be a positive odd integer. Then  $K_{2 m}^\ast$ admits a resolvable directed $m$-cycle decomposition if and only if $m\ge 5$.
\end{conj}

Observe that proving Conjecture~\ref{conj} (which appears to be difficult) would complete the solution to Problem~\ref{prob}. In this paper, we confirm the above conjecture for all $m \le 49$. Thus, we prove the following result.

\begin{theo}\label{the:main}
Let $m$ be an odd integer, $5 \le m \le 49$. Then $K_{2m}^\ast$ admits a resolvable decomposition  into  directed $m$-cycles.
\end{theo}

Except for the smallest case $m=5$,  the above theorem is proved using a general construction that is complemented with a computational result. We expect that with more computing power, this approach can be used to extend our result to even larger values of $m$.

Theorems~\ref{the:BurSaj} and \ref{the:main} immediately yield the following.

\begin{cor}
Let $m$ be an odd integer, $5 \le m \le 49$. Then $K_n^\ast$ admits a resolvable decomposition  into  directed $m$-cycles whenever $n \equiv 0 \pmod{2m}$.
\end{cor}

%%%%%%%%%%%%%%%%%
% PRELIMINARIES %
%%%%%%%%%%%%%%%%%

\section{Preliminaries}

In this paper, the term  {\em digraph} will mean a directed graph with no loops or multiple arcs. The symbol $K_n^\ast$ denotes the {\em complete symmetric digraph} of order $n$; that is, the digraph with $n$ vertices, and with arcs $(u,v)$ and $(v,u)$ for each pair of distinct vertices $u$ and $v$.

For a digraph $D=(V,A)$, a subset $V' \subseteq V$ of its vertex set, and subset $A'\subseteq A$ of its arc set, the symbols $D[V']$ and $D-A'$ will denote the subdigraph of $D$ induced by $V'$, and the subdigraph obtained from $D$ by deleting all arcs in $A'$, respectively. If $D$ is a spanning subdigraph of the complete symmetric digraph $K_n^\ast$ and $A' \subseteq A(K_n^\ast)-A(D)$, then $D+A'$ will denote the digraph $(V(D),A(D) \cup A')$.

A {\em decomposition} of a digraph $D$ is a collection $\{ H_1, H_2, \ldots, H_k \}$ of subdigraphs of $D$ whose arc sets partition the arc set of $D$. If each of the digraphs $H_i$ is isomorphic to a digraph $H$, then $\{ H_1, H_2, \ldots, H_k \}$ is called an {\em $H$-decomposition} of the digraph $G$.

A {\em resolution class} (or {\em parallel class}) of a decomposition ${\cal D}=\{ H_1, H_2, \ldots, H_k \}$ of $D$  is a subset $\{ H_{i_1}, H_{i_2}, \ldots, H_{i_t} \}$ of $\cal D$ with the property that the vertex sets of the digraphs $H_{i_1}, H_{i_2}, \ldots, H_{i_t} $ partition the vertex set of $D$. A decomposition  is called {\em resolvable} if it can be partitioned into resolution classes.

By $\vec{C}_m$ we shall denote the directed cycle of length $m$. The terms  $\vec{C}_m$-decomposition and  resolvable $\vec{C}_m$-decomposition will be abbreviated as $\vec{C}_m$-D and R$\vec{C}_m$-D, respectively.

For a positive integer $m$ and $S \subseteq \ZZ_m^\ast$, the digraph with vertex set $\ZZ_m$ and arc set $\{ (i,i+d): i\in \ZZ_m, d \in S \}$, denoted $\C(m;S)$, is called the {\em directed circulant of order $m$ with connection set $S$}.
The following result, a direct corollary of \cite{BerFavMah}, will be an important ingredient in our constructions.

\begin{lemma}\cite{BerFavMah}\label{lem:BerFavMah}
Let $m$ be a positive integer and $S \subseteq \ZZ_m^*$.  Assume $S$  can be partitioned into sets of the form
\begin{itemize}
\item $\{ d \}$ such that $\gcd(d,m)=1$ and
\item $\{\pm d_i, \pm d_j\}$ such that $\gcd(d_i,d_j,m)=1$.
%(In particular, this holds if  $1 \le d_i<d_j \le \frac{m-1}{2}$ and $d_j-d_i \le 2$ if $3|m$, or $d_j-d_i \le 4$ otherwise.)
\end{itemize}
Then the directed circulant $\C(m;S)$ can be decomposed into directed $m$-cycles.
\end{lemma}

%%%%%%%%%%%%%%%%%%%%%
%%% Constructions %%%
%%%%%%%%%%%%%%%%%%%%%

\section{Results}

\begin{lemma}\label{lem:5-10}
There exists a R$\vec{C}_{5}$-D of $K_{10}^\ast$.
\end{lemma}

\begin{proof}
Label the vertices of $K_{10}^\ast$ by $x_0,x_1,\ldots,x_9$. It can be verified that the following resolution classes (obtained by a computer search) form a R$\vec{C}_{5}$-D of $K_{10}^\ast$.
\begin{eqnarray*}
R_0 &=& \{ x_0 x_1 x_2 x_3 x_4 x_0, x_5 x_6 x_7 x_8 x_9 x_5\} \\
R_1 &=& \{ x_0 x_2 x_1 x_3 x_5 x_0, x_4 x_6 x_8 x_7 x_9 x_4\} \\
R_2 &=& \{ x_0 x_3 x_1 x_4 x_2 x_0, x_5 x_7 x_6 x_9 x_8 x_5\} \\
R_3 &=& \{ x_0 x_4 x_1 x_5 x_8 x_0, x_2 x_6 x_3 x_9 x_7 x_2\} \\
R_4 &=& \{ x_0 x_5 x_2 x_8 x_3 x_0, x_1 x_7 x_4 x_9 x_6 x_1\} \\
R_5 &=& \{ x_0 x_6 x_2 x_5 x_9 x_0, x_1 x_8 x_4 x_3 x_7 x_1\} \\
R_6 &=& \{ x_0 x_7 x_3 x_8 x_6 x_0, x_1 x_9 x_2 x_4 x_5 x_1\} \\
R_7 &=& \{ x_0 x_8 x_2 x_9 x_1 x_0, x_3 x_6 x_4 x_7 x_5 x_3\} \\
R_8 &=& \{ x_0 x_9 x_3 x_2 x_7 x_0, x_1 x_6 x_5 x_4 x_8 x_1\} \\
\end{eqnarray*}
\end{proof}

The rest of the proof of Theorem~\ref{the:main} is divided into two main cases, $m \not\equiv 0 \pmod 3$, which is dealt with in Proposition~\ref{pro:<>0}, and $m \equiv 0 \pmod 3$, which is considered in Proposition~\ref{pro:=0}, as well as two small cases, $m=11$ and $m=9$, which require a modification of the general approach. All of these cases, however, have the following construction in common.

\begin{cons}\label{cons}{\rm
Let $m \ge 5$ be an odd integer, and write $m=2k+1$. Let the vertex set of $D=K_{2m}^\ast$ be $X \cup Y$, where $X=\{x_0,x_1,\ldots,x_{2k} \}$ and $Y=\{ y_0,y_1,\ldots,y_{2k} \}$. We shall call arcs of the form $(x_i, x_{i+d})$ and $(y_i, y_{i+d})$ arcs of {\em pure left} and {\em pure right difference} $d$, respectively, and arcs of the form $(x_i, y_{i+d})$ and $(y_i, x_{i+d})$ arcs of {\em mixed  difference} $d$. All subscripts will be evaluated modulo $m=2k+1$.

Start by defining directed $m$-cycles
$$C_0=x_0 y_0 x_1 y_1 x_2 y_2 \ldots x_k x_0 \qquad \mbox{ and }  \qquad
C_0'=y_k x_{k+1} y_{k+1} \ldots y_{2k} y_k.$$
For $i \in \ZZ_{m}$, obtain $C_i$ and $C_i'$ from $C_0$ and $C_0'$, respectively, by adding $i$ to the subscripts of the vertices in $X$, and $2i$ to the subscripts of the vertices in $Y$, and form resolutions classes
$$R_i= \{ C_i, C_i' \}, \qquad \mbox{ for } i \in \ZZ_{m}.$$
Observe that $R_0,\ldots,R_{m-1}$ use up all arcs of pure left difference $k+1$, all arcs of pure right difference $k+1$, and all arcs of mixed differences except for the arcs
$$(x_{k+i}, y_{k+2i}) \quad \mbox{ and }  \quad (y_{2k+2i}, x_{i}) \quad \mbox{ for all } i \in \ZZ_{m}. \eqno{(\ast)}$$
{\phantom{a} \hfill \framebox[2.2mm]{ } \bigskip}
}
\end{cons}

%\bigskip

Next, we examine the case $m=11$, which requires a  modified construction, but serves as a good introduction to the general approach in the case $m \not\equiv 0 \pmod{3}$ that will be described in Proposition~\ref{pro:<>0}.

\begin{lemma}\label{lem:11}
There exists a R$\vec{C}_{11}$-D of $K_{22}^\ast$.
\end{lemma}

\begin{proof}
With $m=11$, adopt the notation and define resolution classes $R_0,\ldots,R_{10}$ as in Construction~\ref{cons}.
It can be verified that the 22 leftover arcs of mixed differences in ($\ast$) form a directed 22-cycle
$$C=x_5 y_5 \ldots x_5.$$
We decompose $C$ into the following directed paths:
\begin{eqnarray*}
P_1 = &x_5 y_5 \ldots x_6, &
\qquad P_2 =  x_6 y_7,\\
P_3 = &y_7 x_4, &
\qquad P_4 =  x_4 y_3,\\
P_5 = &y_3 \ldots y_2, &
\qquad P_6 =  y_2 x_7, \\
P_7 = &x_7 y_9, &
\qquad P_8 =  y_9 x_5,
\end{eqnarray*}
where $P_1$ and $P_5$ are of length 10 and 6, respectively.
Use the $P_i$ for $i$ odd to form the resolution class
$$R_{11}=\{ P_1 x_6 x_5, P_3 x_4 x_7 P_7 y_9 y_3 P_5 y_2 y_7 \}.$$
We shall use the $P_i$ for $i$ even in the next resolution class. Notice that in $D[Y]$ we have used all arcs of right pure difference 6 and two arcs --- namely, $(y_9, y_3)$ and $(y_2, y_7)$ --- of right pure difference 5. The remaining arcs of right pure difference 5 form a directed $(y_3,y_2)$-path $Q_1'$ of length 2, and a directed $(y_7,y_9)$-path $Q_2'$ of length 7.
If we can find vertex-disjoint directed $(x_7,x_4)$-path of length 7 (call it $Q_1$) and $(x_5,x_6)$-path of length 2 (call it $Q_2$) in $D[X]$, then the next resolution class will be
$$R_{12}=\{ P_2Q_2'P_8Q_2, P_4Q_1'P_6Q_1 \}.$$
What will then remain of $D[Y]$ is a $\C(11;\{\pm 1,\pm 2,\pm 3, \pm 4 \})$, which admits a $\vec{C}_{11}$-D by Lemma~\ref{lem:BerFavMah}.
It thus suffices to appropriately decompose the remaining subdigraph of $D[X]$. In particular, it suffices  to find a set of differences $S \subseteq \ZZ_{11}^\ast$ such that
\begin{description}
\item{($X_1$)} $6 \not\in S$, as left pure difference 6 has already been used;
\item{($X_2$)} $3, 10 \in S$, as only arcs $(x_6,x_5)$ and $(x_4,x_7)$ of these left pure differences have already been used;
\item{($X_3$)} $\C(11;\ZZ_{11}^\ast-S-\{6\})$ admits a decomposition into directed 11-cycles; and
\item{($X_4$)} $\C(11;S)-\{(6,5),(4,7)\}$ admits a decomposition into directed 11-cycles, and vertex-disjoint directed paths: a $(5,6)$-path of length 2 and a $(7,4)$-path of length 7.
\end{description}
Such a set $S$ was found using a computer search. The set $S$, as well as a suitable decomposition, is shown in the appendix.
\end{proof}

\begin{prop}\label{pro:<>0}
Let $m$ be an odd integer such that $m \not\equiv 0 \pmod{3}$, $m \ge 7$, and $m \ne 11$. Let $k=\frac{m-1}{2}$, and define parameters $d,s_i',t_i',s_i,t_i$ (for $i=1,2$) as indicated  below.

\bigskip

\begin{center}
\begin{tabular}{||c||c|c|c|c||}
\hline\hline
Parameter $\setminus$  Case & $k \equiv 0 \pmod{4}$ & $k \equiv 1 \pmod{4}$ & $k \equiv 2 \pmod{4}$ & $k \equiv 3 \pmod{4}$ \\ \hline \hline
$d$ & $(7k+8)/{4}$ & $(5k+7)/{4}$ & $({3k+6})/{4}$ & $({k+5})/{4}$ \\ \hline
$s_1'$ & ${k}/{4}$ & $({3k+1})/{4}$ & $({5k+2})/{4}$ & $({7k+3})/{4}$ \\ \hline
%$t_1'$ & $k$ & $k$ & $k$ & $k$ \\ \hline
$s_2'$ & $({3k+4})/{4}$ & $({k+3})/{4}$ & $({7k+6})/{4}$ & $({5k+5})/{4}$ \\ \hline
$t_2'$ & $({k-2})/{2}$ & $({3k-1})/{2}$ & $({k-2})/{2}$ & $({3k-1})/{2}$ \\ \hline
%$s_1$ & $2k-1$ & $2k-1$ & $2k-1$ & $2k-1$\\ \hline
$t_1$ & $({3k+2})/{2}$ & $({k+1})/{2}$ & $({3k+2})/{2}$ & $({k+1})/{2}$\\ %\hline
%$s_2$ & $k$ & $k$ & $k$ & $k$\\ \hline
%$t_2$ & $({k-2})/{2}$ & $({3k-1})/{2}$ & $({k-2})/{2}$ & $({3k-1})/{2}$\\
\hline\hline
\end{tabular}
\end{center}
In addition, let $t_1'=s_2=k$, $s_1=2k-1$, and $t_2=t_2'$.

\bigskip

Then $\gcd(d,m)=1$, and hence for each $i=1,2$, there exists a unique $r_i \in \ZZ_m$ such that $s_i'+r_i d = t_i'$ (in $\ZZ_m$).
Furthermore, define $a_i=(t_i,s_i)$ and $d_i^Y=s_i-t_i$ (in $\ZZ_m$).

Now assume there exists a set $S \subseteq \ZZ_m^\ast$ such that:
\begin{description}
\item{($Y_1$)} $k+1 \not\in S$;
\item{($Y_2$)} $d_1^Y, d_2^Y \in S$;
\item{($Y_3$)} $\C(m;\ZZ_m^\ast-(S \cup \{ k+1 \}))$ admits a $\vec{C}_m$-D; and
\item{($Y_4$)} $\C(m;S)-\{ a_1,a_2 \}$ admits a decomposition into directed $m$-cycles and two vertex-disjoint directed paths: an $(s_1,t_1)$-path of length $r_1$ and an $(s_2,t_2)$-path of length $r_2$.
\end{description}
Then $K_{2m}^\ast$ admits a R$\vec{C}_m$-D.
\end{prop}

\begin{proof}
Adopt the notation and define resolution classes $R_0,\ldots,R_{m-1}$ as in Construction~\ref{cons}. It can be verified that, since $m \not\equiv 0 \pmod{3}$, the $2m$ leftover arcs of mixed differences in ($\ast$) form a directed $2m$-cycle
$$C=y_k \ldots x_k y_k.$$
Write $C=P_1 P_2 \ldots P_8$ as a concatenation of directed paths such that $P_1$ is of length $m-1$, $P_5$ is of length $m-5$, and the rest are of length 1. It can be shown for each congruency class of $k$ modulo 4 that the paths are
\begin{eqnarray*}
P_1 = &y_{s_2} \ldots y_{t_2}, &
\qquad P_2 =  y_{t_2} x_{s_1'},\\
P_3 = &x_{s_1'} y_{t_1}, &
\qquad P_4 =  y_{t_1} x_{s_2'},\\
P_5 = &x_{s_2'}  \ldots x_{t_2'}, &
\qquad P_6 =  x_{t_2'} y_{s_1}, \\
P_7 = &y_{s_1} x_{t_1'}, &
\qquad P_8 =  x_{t_1'} y_{s_2},
\end{eqnarray*}
where the parameters $s_i,t_i,s_i',t_i'$ (for $i=1,2$) are as defined in the statement of the proposition.
We  use the $P_i$ for $i$ odd, together with 4 linking arcs (two of pure left, and two of pure right difference) to form the resolution class

$$R_{m}=\{ P_1 y_{t_2} y_{s_2}, P_5 x_{t_2'} x_{s_1'} P_3 y_{t_1} y_{s_1} P_7 x_{t_1'} x_{s_2'} \}.$$
The linking arcs are:
$$(x_{t_2'}, x_{s_1'}) \mbox{ and } (x_{t_1'}, x_{s_2'}) \mbox{ of pure left difference } d=s_1'-t_2'=s_2'-t_1',$$
$$a_1=(y_{t_1}, y_{s_1}) \mbox{ of pure right difference }  d_1^Y=s_1-t_1, \mbox{ and}$$
$$
a_2=(y_{t_2}, y_{s_2}) \mbox{ of pure right difference }  d_2^Y=s_2-t_2,$$
with $d, d_1^Y, d_2^Y$ as defined in the statement of the proposition.
Since $m\ne 11$, observe that none of these pure differences are equal to $k+1$ (which has already been used in $R_0, \ldots, R_{m-1}$).

The $P_i$ for $i$ even will be used in the next resolution class as follows. Since $m \not\equiv 0 \pmod{3}$, it can be shown that $\gcd(2k+1,d)=1$.  Therefore, the arcs of pure left difference $d$ form a directed $m$-cycle, and in particular, those that have not been used in $R_m$ form a directed $(x_{s_1'}, x_{t_1'})$-path $Q_1'$ of length $r_1$ and a directed $(x_{s_2'}, x_{t_2'})$-path $Q_2'$ of length $r_2$, where $r_1$ and $r_2$ are as defined in the statement of the proposition.

Now let $S$ be a subset of $\ZZ_m^\ast$ satisfying Conditions ($Y_1$)--($Y_4$) of the proposition, and let $Q_1$ and $Q_2$ be the corresponding vertex-disjoint directed
$(y_{s_1},y_{t_1})$-path of length $r_1$ and $(y_{s_2},y_{t_2})$-path of length $r_2$, respectively. We then let the next resolution class be
$$R_{m+1}= \{ P_2Q_1'P_8Q_2, P_4Q_2'P_6Q_1 \}.$$
All arcs of mixed differences have now been used in resolution classes $R_0,\ldots,R_{m+1}$.
In $D[X]$, we have also used up all arcs of differences $k+1$ and $d$. Since $\gcd(2k+1,k+1)=\gcd(2k+1,d)=1$,  Lemma~\ref{lem:BerFavMah} now guarantees that the remaining subdigraph of $D[X]$ admits a $\vec{C}_m$-D.

In $D[Y]$, however, we have used up:
\begin{itemize}
\item  all arcs of difference $k+1$;
\item arcs $a_1$ and $a_2$  of differences $d_1^Y$ and $d_2^Y$, respectively; and
\item arcs used in the directed paths $Q_1$ and $Q_2$.
\end{itemize}
Assumptions ($Y_1$)--($Y_4$) now guarantee that the remaining subdigraph of $D[Y]$ admits a $\vec{C}_m$-D.
Finally, the directed $m$-cycles from the remaining subdigraphs of $D[X]$ and $D[Y]$ can be arranged into resolution classes that complete our R$\vec{C}_m$-D of $K_{2m}^\ast$.
\end{proof}

We now turn our attention to the case $m \equiv 0 \pmod{3}$. As before, a small case ($m=9$) requires a modified construction and will also serve as an introduction to the general approach.

\begin{lemma}\label{lem:9}
There exists a R$\vec{C}_{9}$-D of $K_{18}^\ast$.
\end{lemma}

\begin{proof}
Adopt the notation and construction of resolution classes $R_0,\ldots,R_8$ from Construction~\ref{cons}.
The $18$ leftover arcs of mixed differences from ($\ast$) now form three directed $6$-cycles, which we write as a concatenation of directed paths of length $2$ and linking arcs as follows:
\begin{eqnarray*}
C_{(1)} &=& x_{0} y_{5} x_3 y_2 x_6 y_8 x_{0}= P_{1}^X x_{3} y_{2} P_{1}^Y y_{8}x_0, \\
C_{(2)} &=& x_{1} y_{7} x_4 y_4 x_7 y_1 x_{1}=P_{2}^X x_{4} y_{4} P_{2}^Y y_{1}x_1, \\
C_{(3)} &=& x_{2} y_{0} x_5 y_6 x_8 y_{3} x_{2}=P_{3}^X x_{5} y_{6} P_{3}^Y y_{3}x_2.
\end{eqnarray*}

We use the directed paths $P_{i}^X, P_{i}^Y$ (for $i=1,2,3$), together with 6 linking arcs of pure differences, to form the resolution class $R_{9}$:
$$R_9=\{ P_{1}^X x_{3} x_1 P_{2}^X x_{4} x_2 P_{3}^X x_{5} x_0, P_{1}^Y y_{8}y_{6} P_{3}^Y y_{3}y_{4} P_{2}^Y y_{1}y_{2} \}.$$
We have thus used the following linking arcs:
\begin{eqnarray*}
b^X_1 &=& (x_{3},x_1) \quad \mbox{ of pure left difference } d_1^X=7, \\
b^X_2 &=& (x_{4},x_2) \quad \mbox{ of pure left difference } d_1^X=7,  \\
b^X_3 &=& (x_{5},x_0) \quad \mbox{ of pure left difference } d_2^X=4, \\
b^Y_1 &=& (y_{1},y_{2}) \quad \mbox{ of pure right difference }   d_1^Y=1, \\
b^Y_2 &=& (y_{3},y_{4}) \quad \mbox{ of pure right difference }   d_1^Y=1, \\
b^Y_3 &=& (y_{8},y_{6}) \quad \mbox{ of pure right difference } d_2^Y=7.
\end{eqnarray*}
Note that none of these differences are equal to $5$, which has been used in $R_0,\ldots,R_{8}$.

We have now used up all arcs of mixed differences except for the arcs
$(x_3,y_2),(x_4,y_4)$, $(x_5,y_6)$ and arcs $(y_{8},x_0), (y_1,x_1),(y_3,x_2)$.

To form the resolution class $R_{10}$, we want to find three vertex-disjoint directed paths with sources $x_0, x_1, x_2$ and terminals $x_{3}, x_{4}, x_{5}$ using some of the remaining arcs in $D[X]$, and three vertex-disjoint directed paths with sources $y_2,y_4,y_6$ and terminals $y_{8}, y_1, y_3$ using some of the remaining arcs in $D[Y]$; these paths, together with all the remaining arcs of mixed differences, will form two vertex-disjoint directed $9$-cycles. In particular, we can define
$$R_{10}=\{ Q_1'x_3y_2Q_1y_3x_2Q_2'x_4y_4Q_2y_1x_1, Q_3'x_5y_6Q_3y_8x_0 \}$$
as long as we have suitable directed paths
\begin{eqnarray*}
Q_1': && (x_1,x_3) \mbox{-path of length } 1, \\
Q_2': && (x_2,x_4) \mbox{-path of length } 1, \\
Q_3': && (x_0,x_5) \mbox{-path of length } 4, \\
Q_1: && (y_2,y_3) \mbox{-path of length } 1, \\
Q_2: && (y_4,y_1) \mbox{-path of length } 2, \mbox{ and} \\
Q_3: && (y_6,y_8) \mbox{-path of length } 3
\end{eqnarray*}
that use only hitherto unused arcs of pure differences.
More precisely, it suffices to find sets $S^X, S^Y \subseteq \ZZ_9^\ast$ such that the following hold.
\begin{description}
\item{($X_1$)} $5 \not\in S^X$, as left pure difference 5 has already been used;
\item{($X_2$)} $4,7 \in S^X$, as arcs $(x_3,x_1),(x_4,x_2),(x_5,x_0)$ have already been used;
\item{($X_3$)} $\C(9;\ZZ_{9}^\ast-S^X-\{5\})$ admits a decomposition into directed 9-cycles; and
\item{($X_4$)} $\C(9;S^X)-\{(3,1),(4,2),(5,0)\}$ admits a decomposition into directed 9-cycles and pairwise vertex-disjoint directed $(1,3)$-path of length 1, $(2,4)$-path of length 1, and $(0,5)$-path of length 4;
\item{($Y_1$)} $5 \not\in S^Y$, as right pure difference 5 has already been used;
\item{($Y_2$)} $1,7 \in S^Y$, as arcs $(y_1,y_2),(y_3,y_4),(y_8,y_6)$ have already been used;
\item{($Y_3$)} $\C(9;\ZZ_{9}^\ast-S^Y-\{5\})$ admits a decomposition into directed 9-cycles; and
\item{($Y_4$)} $\C(9;S^Y)-\{(1,2),(3,4),(8,6)\}$ admits a decomposition into directed 9-cycles and pairwise vertex-disjoint directed paths: a $(2,3)$-path of length 1, a $(4,1)$-path of length 2, and a $(6,8)$-path of length 3.
\end{description}
Such sets $S^X$ and $S^Y$ were found using a computer search. These sets, as well as suitable decompositions, are shown in the appendix.
\end{proof}

\begin{prop}\label{pro:=0}
Let $m$ be an odd integer such that $m \equiv 0 \pmod{3}$, $m \ge 15$. Let $k=\frac{m-1}{2}$, and define parameters $s_1$ and $t_1$  as indicated in the table below.

\bigskip

\begin{center}
\begin{tabular}{||c||c|c|c|c||}
\hline\hline
Parameter $\setminus$  Case & $k \equiv 0 \pmod{4}$ & $k \equiv 1 \pmod{4}$ & $k \equiv 2 \pmod{4}$ & $k \equiv 3 \pmod{4}$ \\ \hline \hline
$s_1$ & ${k}/{2}$ & $({3k+1})/{2}$ & ${k}/{2}$ & $({3k+1})/{2}$ \\ \hline
$t_1$ & ${3k}/{4}$ & $({k-1})/{4}$ & $({7k+2})/{4}$ & $({5k+1})/{4}$ \\
\hline\hline
\end{tabular}
\end{center}
\noindent In addition, for $i=1,2$, let $s_{1+i}=s_1+2i$ and $t_{1+i}=t_1+i$ (all evaluated in $\ZZ_m$).

Furthermore, define arcs:
\begin{eqnarray*}
b_1^X =(t_1,1), \qquad & \qquad b_1^Y =(1,s_1), \qquad & \qquad c_1 =(t_1,0),\\
b_2^X =(t_2,2), \qquad & \qquad b_2^Y =(3,s_2), \qquad & \qquad c_2 =(t_2,1),\\
b_3^X =(t_3,0), \qquad & \qquad \;\; b_3^Y =(-1,s_3), \qquad & \qquad c_3 =(t_3,2).
\end{eqnarray*}
Now assume there exist sets $S^X,S^Y \subseteq \ZZ_m^\ast$ such that:
\begin{description}
\item{($X_1$)} $k+1,-t_1 \not\in S^X$;
\item{($X_2$)} $1-t_1, -2-t_1 \in S^X$;
\item{($X_3$)} $\C(m;\ZZ_m^\ast-(S^X \cup \{ k+1,-t_1 \}))$ admits a $\vec{C}_m$-D;
\item{($X_4$)} $\C(m;S^X)-\{ b_1^X,b_2^X,b_3^X \}+\{c_1,c_2,c_3 \}$ admits a $\vec{C}_m$-D;
\item{($Y_1$)} $k+1 \not\in S^Y$;
\item{($Y_2$)} $s_1-1, s_1+5 \in S^Y$;
\item{($Y_3$)} $\C(m;\ZZ_m^\ast-(S^Y \cup \{ k+1 \}))$ admits a $\vec{C}_m$-D; and
\item{($Y_4$)} $\C(m;S^Y)-\{ b_1^Y,b_2^Y,b_3^Y \}$ admits a decomposition into directed $m$-cycles and three pairwise vertex-disjoint directed paths: an $(s_1,-1)$-path of length $\frac{2m}{3}-1$, an $(s_2,3)$-path of some length $q \in \{ 1,\ldots,\frac{m}{3}-3 \}$, and an $(s_3,1)$-path of length $\frac{m}{3}-2-q$.
\end{description}
Then $K_{2m}^\ast$ admits a R$\vec{C}_m$-D.
\end{prop}

\begin{proof}
Adopt the notation and construction of resolution classes $R_0,\ldots,R_{m-1}$ from Construction~\ref{cons}.
It can be verified that, since $m \equiv 0 \pmod{3}$, the $2m$ remaining arcs of mixed differences in ($\ast$) form three directed $\frac{2m}{3}$-cycles, which we write as a concatenation of directed paths of length $\frac{m}{3}-1$ and linking arcs as follows:
\begin{eqnarray*}
C_{(1)} &=& x_{0} y_{k+1} \ldots y_{-1} x_{0}= P_{1}^X x_{t_1} y_{s_1} P_{1}^Y y_{-1}x_0, \\
C_{(2)} &=& x_{1} y_{k+3} \ldots y_{1} x_{1}=P_{2}^X x_{t_2} y_{s_2} P_{2}^Y y_{1}x_1, \\
C_{(3)} &=& x_{2} y_{k+5} \ldots y_{3} x_{2}=P_{3}^X x_{t_3} y_{s_3} P_{3}^Y y_{3}x_2.
\end{eqnarray*}
It can be verified that, for each congruency class of $k$ modulo 4, the parameters $s_i,t_i$ (for $i=1,2,3$) have values as defined in the statement of the proposition.

We use the directed paths $P_{i}^X, P_{i}^Y$ (for $i=1,2,3$), together with 6 linking arcs of pure differences, to form the resolution class $R_{m}$:
$$R_m=\{ P_{1}^X x_{t_1} x_1 P_{2}^X x_{t_2} x_2 P_{3}^X x_{t_3} x_0, P_{1}^Y y_{-1}y_{s_3} P_{3}^Y y_{3}y_{s_2} P_{2}^Y y_{1}y_{s_1} \}.$$
We have thus used the following linking arcs:
\begin{eqnarray*}
b^X_1 &=& (x_{t_1},x_1) \quad \mbox{ of pure left difference } d_1^X=1-t_1, \\
b^X_2 &=& (x_{t_2},x_2) \quad \mbox{ of pure left difference } d_1^X=1-t_1,  \\
b^X_3 &=& (x_{t_3},x_0) \quad \mbox{ of pure left difference } d_2^X=-2-t_1, \\
b^Y_1 &=& (y_{1},y_{s_1}) \quad \mbox{ of pure right difference }   d_1^Y=s_1-1, \\
b^Y_2 &=& (y_{3},y_{s_2}) \quad \mbox{ of pure right difference }   d_1^Y=s_1-1, \\
b^Y_3 &=& (y_{-1},y_{s_3}) \quad \mbox{ of pure right difference } d_2^Y=s_1+5.
\end{eqnarray*}
Note that, in all cases, none of these differences are equal to $k+1$.

We have now used up all arcs of mixed differences except for the arcs $(x_{t_i}, y_{s_i})$ for $i=1,2,3$, and arcs $(y_{-1},x_0), (y_1,x_1),(y_3,x_2)$.

To form the resolution class $R_{m+1}$, we want to find three vertex-disjoint directed paths of appropriate lengths with sources $x_0, x_1, x_2$ and terminals $x_{t_1}, x_{t_2}, x_{t_3}$ using some of the remaining arcs in $D[X]$, and three vertex-disjoint directed paths with sources $y_{s_1},y_{s_2},y_{s_3}$ and terminals $y_{-1}, y_1, y_3$ using some of the remaining arcs in $D[Y]$; these paths, together with all the remaining arcs of mixed differences, will form two vertex-disjoint directed $m$-cycles.

It can be shown in each case that $\gcd(m,t_1)=3$, so the following are indeed directed $(\frac{m}{3}-1)$-paths in $D[X]$ with the required sources and terminals:
\begin{eqnarray*}
Q_1' &=& x_0 x_{-t_1} x_{-2t_1} \ldots x_{t_1}, \\
Q_2' &=& x_1 x_{1-t_1} x_{1-2t_1} \ldots x_{t_2}, \mbox{ and} \\
Q_3' &=& x_2 x_{2-t_1} x_{2-2t_1} \ldots x_{t_3}.
\end{eqnarray*}
Observe that these paths use all arcs of difference $d^X=-t_1$ except for arcs $c_1=(x_{t_1}, x_0)$, $c_2=(x_{t_2}, x_1)$, and $c_3=(x_{t_3}, x_2)$.

Now let $S^X,S^Y \subseteq \ZZ_m^\ast$ be two sets satisfying Assumptions ($X_1$)--($X_4$),($Y_1$)--($Y_4$) of the proposition. Furthermore, let $Q_1,Q_2,Q_3$ be the pairwise vertex-disjoint directed paths in $D[Y]$ whose existence is assured by Condition ($Y_4$), so that
$$Q_1 \mbox{ is a directed } (y_{s_1},y_{-1}) \mbox{-path of length } \textstyle{\frac{2m}{3}-1}, $$
$$Q_2  \mbox{ is a directed } (y_{s_2},y_3) \mbox{-path of length } q, \mbox{ for some } \textstyle{q \in \{ 1,\ldots,\frac{m}{3}-3 \}}, \mbox{ and}$$
$$Q_3 \mbox{ is a directed } (y_{s_3},y_1) \mbox{-path of length } \textstyle{\frac{m}{3}-2-q}.$$
We may then define our next resolution class as
$$R_{m+1}=\{ Q_1'x_{t_1}y_{s_1} Q_1y_{-1}x_0,
Q_2'x_{t_2}y_{s_2}Q_2y_{3}x_2Q_3'x_{t_3}y_{s_3}Q_3y_{1}x_1 \}.$$
Now, all arcs of mixed differences have been used in resolution classes $R_1,\ldots,R_{m+1}$. In addition, we have also used up in $D[X]$:
\begin{itemize}
\item  all arcs of difference $k+1$;
\item arcs $b_i^X$, for $i=1,2,3$ (of differences $1-t_1$ and $-2-t_1$); and
\item all arcs of difference $-t_1$ except $c_i$, for $i=1,2,3$.
\end{itemize}
Assumptions ($X_1$)--($X_4$) now guarantee that the remaining subdigraph of $D[X]$ admits a $\vec{C}_m$-D.
In $D[Y]$, however, we have used up:
\begin{itemize}
\item  all arcs of difference $k+1$;
\item arcs $b_i^Y$, for $i=1,2,3$  (of differences $s_1-1$ and $s_1+5$); and
\item arcs used in the directed paths $Q_i$, for $i=1,2,3$.
\end{itemize}
Assumptions ($Y_1$)--($Y_4$) now guarantee that the remaining subdigraph of $D[Y]$ admits a $\vec{C}_m$-D.
The directed $m$-cycles from the remaining subdigraphs of $D[X]$ and $D[Y]$ can be arranged into resolution classes that complete our R$\vec{C}_m$-D of $K_{2m}^\ast$.
\end{proof}

\begin{proofof} \ref{the:main}.
Let $m$ be an odd integer, $5 \le m \le 49$. Then $K_{2m}^\ast$ admits a R$\vec{C}_m$-D by Lemma~\ref{lem:5-10} if $m=5$, by Lemma~\ref{lem:11} if $m=11$, and by Lemma~\ref{lem:9} if $m=9$. It can be verified that the computational results in Appendix A show that the conditions of Proposition~\ref{pro:<>0} hold for all odd $m$, $7 \le m \le 49$, $m \not\equiv 0 \pmod{3}$, $m \ne 11$; hence $K_{2m}^\ast$ admits a R$\vec{C}_m$-D for all such $m$. Finally, Appendix B shows that the conditions of Proposition~\ref{pro:=0} hold for all odd $m$, $15 \le m \le 45$, $m \equiv 0 \pmod{3}$; hence $K_{2m}^\ast$ admits a R$\vec{C}_m$-D for all such $m$ as well. Therefore, the statement holds for all odd $m$, $5 \le m \le 49$.
\end{proofof}

\subsection*{Acknowledgements}
The authors gratefully acknowledge support by the Natural Sciences and Engineering Research Council of Canada. Special thanks to Patrick Niesink and Ryan Murray, who wrote most of the code used to obtain the computational results, and to Aras Erzurumluo\u{g}lu for careful proofreading of the manuscript.

%\end{document}

\appendix

\section{Computational results --- Case $m \not\equiv 0 \pmod{3}$}

For each value of $m$ we give a set $S \subseteq \ZZ_m^\ast$ satisfying Conditions ($Y_1$) -- ($Y_4$) of Proposition~\ref{pro:<>0} (if $m \ne 11$), or Conditions ($X_1$) -- ($X_4$) from the proof of Lemma~\ref{lem:11} (if $m=11$). The required differences appear in bold type. In addition, we give a desired decomposition into directed $m$-cycles $C_i$ and vertex-disjoint directed paths $Q_1$ and $Q_2$. If $m$ is not prime, we also give a partition of $\ZZ_m^\ast-(S \cup \{ \frac{m+1}{2} \})$ satisfying the assumptions of Lemma~\ref{lem:BerFavMah}.

\begin{itemize}
\item $m=7$ \\
$S =\{ 2,\mathbf{3,6} \}$ \\
$Q_1 = (5, 0, 2)$ \\
$Q_2 = (3, 6, 1, 4)$ \\
$C_1 =(0,3,5,4,6,2,1,0)$ \\
$C_2 =(0,6,5,1,3,2,4,0)$

\item $m=11$ \\
$S=\{ \mathbf{3},4,9, \mathbf{10} \}$ \\
$Q_1=(7, 10, 9, 2, 0, 3, 1, 4)$ \\
$Q_2=(5, 8, 6)$ \\
$C_1=(0, 10, 2, 6, 9, 8, 1, 5, 4, 3, 7, 0)$ \\
$C_2=(0, 4, 8, 7, 6, 10, 3, 2, 5, 9, 1, 0)$ \\
$C_3=(0, 9, 7, 5, 3, 6, 4, 2, 1, 10, 8, 0)$

\item $m=13$ \\
$S=\{ \mathbf{1}, 2, 3, \mathbf{4} \}$ \\
$Q_1=(11, 1, 5, 7, 10)$ \\
$Q_2=(6, 9, 0, 3, 4, 8, 12, 2)$ \\
$C_1=(0, 1, 2, 4, 5, 6, 8, 9, 10, 12, 3, 7, 11, 0)$ \\
$C_2=(0, 4, 7, 8, 11, 2, 5, 9, 12, 1, 3, 6, 10, 0)$ \\
$C_3=(0, 2, 3, 5, 8, 10, 1, 4, 6, 7, 9, 11, 12, 0)$

\item $m=17$ \\
$S=\{ 1,\mathbf{2}, 3, \mathbf{5} \}$ \\
$Q_1=(15, 16, 1, 4, 7, 9, 14, 2, 5, 6, 11, 13)$ \\
$Q_2=(8, 10, 12, 0, 3)$ \\
$C_1=(0, 2, 4, 6, 8, 9, 12, 13, 14, 15, 1, 3, 5, 7, 10, 11, 16, 0)$ \\
$C_2=(0, 5, 8, 11, 14, 16, 2, 3, 4, 9, 10, 13, 1, 6, 7, 12, 15, 0)$ \\
$C_3=(0, 1, 2, 7, 8, 13, 16, 4, 5, 10, 15, 3, 6, 9, 11, 12, 14, 0)$

\item $m=19$ \\
$S=\{ 2,\mathbf{12}, \mathbf{15}\}$ \\
$Q_1=(17, 0, 15, 11, 7, 3, 5)$ \\
$Q_2=(9, 2, 4, 6, 8, 10, 12, 14, 16, 18, 1, 13)$ \\
$C_1=(0, 12, 8, 4, 16, 9, 5, 1, 3, 18, 11, 13, 15, 17, 10, 6, 2, 14, 7, 0)$ \\
$C_2=(0, 2, 17, 13, 6, 18, 14, 10, 3, 15, 8, 1, 16, 12, 5, 7, 9, 11, 4, 0)$

\item $m=23$ \\
$S=\{ 1,2,\mathbf{15, 18} \}$ \\
$Q_1=(21, 22, 17, 9, 1, 19, 20, 12, 7, 8, 10, 5, 0, 2, 4, 6)$ \\
$Q_2=(11, 3, 18, 13, 14, 15, 16)$ \\
$C_1=(0, 15, 7, 22, 14, 9, 4, 19, 11, 6, 1, 2, 3, 5, 20, 21, 16, 17, 18, 10, 12, 13, 8, 0)$ \\
$C_2=(0, 18, 19, 14, 16, 8, 9, 10, 11, 13, 15, 17, 12, 4, 5, 6, 7, 2, 20, 22, 1, 3, 21, 0)$ \\
$C_3=(0, 1, 16, 18, 20, 15, 10, 2, 17, 19, 21, 13, 5, 7, 9, 11, 12, 14, 6, 8, 3, 4, 22, 0)$

\item $m=25$ \\
$S=\{ 1, 2, \mathbf{4, 7}\} $ \\
$Q_1=(23, 2, 6, 10, 14, 15, 16, 17, 19)$ \\
$Q_2=(12, 13, 20, 21, 0, 7, 8, 9, 11, 18, 22, 24, 1, 3, 4, 5)$ \\
$C_1=(0, 4, 8, 12, 16, 20, 24, 6, 7, 11, 15, 19, 1, 2, 3, 10, 17, 21, 22, 23, 5, 9, 13, 14, 18, 0)$ \\
$C_2=(0, 1, 5, 6, 8, 15, 22, 4, 11, 13, 17, 24, 3, 7, 14, 16, 18, 20, 2, 9, 10, 12, 19, 21, 23, 0)$ \\
$C_3=(0, 2, 4, 6, 13, 15, 17, 18, 19, 20, 22, 1, 8, 10, 11, 12, 14, 21, 3, 5, 7, 9, 16, 23, 24, 0)$ \\
Partition contains:  $\{\pm 3,\pm 5\}$, $\{\pm 6,\pm 10\}$, and $\{ e \}$ for each remaining difference $e$

\item $m=29$ \\
$S=\{ 1, 2, \mathbf{5, 8}\} $ \\
$Q_1=(27, 28, 7, 8, 9, 11, 16, 21, 23, 25, 26, 5, 10, 12, 13, 15, 17, 18, 20, 22)$ \\
$Q_2=(14, 19, 24, 0, 1, 2, 3, 4, 6)$ \\
$C_1=(0, 5, 13, 18, 23, 28, 4, 9, 14, 22, 1, 6, 7, 15, 20, 25, 27, 3, 8, 16, 24, 26, 2, 10, 11, 12, 17,$ $ 19, 21, 0)$ \\
$C_2=(0, 8, 10, 15, 23, 2, 7, 9, 17, 22, 24, 25, 4, 12, 20, 21, 26, 28, 1, 3, 5, 6, 11, 13, 14, 16, 18,$ $ 19, 27, 0)$ \\
$C_3=(0, 2, 4, 5, 7, 12, 14, 15, 16, 17, 25, 1, 9, 10, 18, 26, 27, 6, 8, 13, 21, 22, 23, 24, 3, 11, 19,$ $ 20, 28, 0)$

\item $m=31$ \\
$S=\{1, \mathbf{21, 24} \}$ \\
$Q_1=(29, 19, 9, 30, 23, 13, 14, 4, 28, 18, 8)$ \\
$Q_2=(15, 16, 17, 10, 11, 12, 5, 6, 7, 0, 1, 2, 3, 24, 25, 26, 27, 20, 21, 22)$ \\
$C_1=(0, 21, 11, 4, 5, 26, 16, 9, 10, 3, 27, 17, 7, 28, 29, 22, 12, 2, 23, 24, 14, 15, 8, 1, 25, 18, 19,$ $ 20, 13, 6, 30, 0)$ \\
$C_2=(0, 24, 17, 18, 11, 1, 22, 23, 16, 6, 27, 28, 21, 14, 7, 8, 9, 2, 26, 19, 12, 13, 3, 4, 25, 15, 5,$ $ 29, 30, 20, 10, 0)$

\item $m=35$ \\
$S=\{ 1, \mathbf{24, 27} \}$ \\
$Q_1=(33, 22, 14, 3, 27, 16, 5, 32, 21, 13, 2, 29, 18, 10, 34, 26, 15, 7, 8, 0, 1, 28, 20, 9)$ \\
$Q_2=(17, 6, 30, 19, 11, 12, 4, 31, 23, 24, 25)$ \\
$C_1=(0, 24, 16, 8, 9, 1, 25, 26, 27, 28, 17, 18, 19, 20, 12, 13, 14, 6, 7, 31, 32, 33, 34, 23, 15, 4,$ $ 5, 29, 21, 10, 2, 3, 30, 22, 11, 0)$ \\
$C_2=(0, 27, 19, 8, 32, 24, 13, 5, 6, 33, 25, 14, 15, 16, 17, 9, 10, 11, 3, 4, 28, 29, 30, 31, 20, 21,$ $ 22, 23, 12, 1, 2, 26, 18, 7, 34, 0)$ \\
Partition contains:  $\{\pm 5,\pm 7\}$, $\{\pm 10,\pm 14\}$, $\{\pm 15,\pm 2\}$, and $\{ e \}$ for each remaining difference $e$

\item $m=37$ \\
$S=\{ 1, \mathbf{7, 10} \}$ \\
$Q_1=(35, 36, 0, 1, 11, 12, 13, 14, 15, 25, 26, 27, 28)$ \\
$Q_2=(18, 19, 29, 2, 9, 10, 20, 21, 22, 23, 30, 3, 4, 5, 6, 16, 17, 24, 31, 32, 33, 34, 7, 8)$ \\
$C_1=(0, 7, 14, 21, 28, 1, 8, 15, 22, 29, 36, 9, 16, 26, 33, 6, 13, 23, 24, 34, 35, 5, 12, 19, 20, 30,$ $31, 4, 11, 18, 25, 32, 2, 3, 10, 17, 27, 0)$ \\
$C_2=(0, 10, 11, 21, 31, 1, 2, 12, 22, 32, 5, 15, 16, 23, 33, 3, 13, 20, 27, 34, 4, 14, 24, 25, 35, 8, 9,$ $ 19, 26, 36, 6, 7, 17, 18, 28, 29, 30, 0)$

\item $m=41$ \\
$S=\{ 1, \mathbf{8, 11} \}$ \\
$Q_1=(39, 6, 14, 15, 23, 24, 32, 40, 7, 18, 26, 34, 1, 2, 10, 11, 19, 27, 35, 36, 3, 4, 12, 13, 21, 22,$ $ 30, 31)$ \\
$Q_2=(20, 28, 29, 37, 38, 5, 16, 17, 25, 33, 0, 8, 9)$ \\
$C_1=(0, 11, 12, 23, 31, 1, 9, 17, 28, 36, 6, 7, 8, 19, 20, 21, 32, 2, 3, 14, 22, 33, 34, 35, 5, 13, 24,$ $ 25, 26, 37, 4, 15, 16, 27, 38, 39, 40, 10, 18, 29, 30, 0)$ \\
$C_2=(0, 1, 12, 20, 31, 32, 33, 3, 11, 22, 23, 34, 4, 5, 6, 17, 18, 19, 30, 38, 8, 16, 24, 35, 2, 13, 14,$ $  25,36, 37, 7, 15, 26, 27, 28, 39, 9, 10, 21, 29, 40, 0)$

\item $m=43$ \\
$S=\{ 1, \mathbf{30, 33} \}$ \\
$Q_1=(41, 28, 15, 2, 32, 19, 6, 36, 23, 10, 0, 33, 34, 24, 11)$ \\
$Q_2=(21, 22, 12, 13, 3, 4, 5, 35, 25, 26, 16, 17, 18, 8, 9, 42, 29, 30, 20, 7, 37, 38, 39, 40, 27, 14,$ $ 1, 31)$ \\
$C_1=(0, 30, 31, 32, 33, 20, 10, 11, 1, 34, 21, 8, 38, 28, 18, 19, 9, 39, 29, 16, 3, 36, 26, 27, 17, 4,$ $ 37, 24, 25, 12, 2, 35, 22, 23, 13, 14, 15, 5, 6, 7, 40, 41, 42, 0)$ \\
$C_2=(0, 1, 2, 3, 33, 23, 24, 14, 4, 34, 35, 36, 37, 27, 28, 29, 19, 20, 21, 11, 12, 42, 32, 22, 9, 10,$ $ 40, 30, 17, 7, 8, 41, 31, 18, 5, 38, 25, 15, 16, 6, 39, 26, 13, 0)$

\item $m=47$ \\
$S=\{ 1, \mathbf{33, 36} \}$ \\
$Q_1=(45, 46, 35, 24, 13, 2, 3, 4, 40, 29, 18, 19, 5, 41, 27, 16, 17, 6, 7, 43, 32, 33, 22, 8, 44, 30,$ $31, 20, 21, 10, 11, 12)$ \\
$Q_2=(23, 9, 42, 28, 14, 0, 36, 37, 38, 39, 25, 26, 15, 1, 34)$ \\
$C_1=(0, 33, 34, 20, 6, 42, 43, 29, 15, 16, 2, 35, 36, 22, 23, 12, 1, 37, 26, 27, 13, 14, 3, 39, 28, 17,$ $ 18, 4, 5, 38, 24, 25, 11, 44, 45, 31, 32, 21, 7, 40, 41, 30, 19, 8, 9, 10, 46, 0)$ \\
$C_2=(0, 1, 2, 38, 27, 28, 29, 30, 16, 5, 6, 39, 40, 26, 12, 13, 46, 32, 18, 7, 8, 41, 42, 31, 17, 3, 36,$ $ 25, 14, 15, 4, 37, 23, 24, 10, 43, 44, 33, 19, 20, 9, 45, 34, 35, 21, 22, 11, 0)$

\item $m=49$ \\
$S=\{ 2, \mathbf{10, 13} \}$ \\
$Q_1=(47, 8, 18, 28, 38, 48, 9, 19, 21, 31, 44, 46, 10, 23, 25, 35, 37)$ \\
$Q_2=(24, 34, 36, 0, 2, 12, 22, 32, 45, 6, 16, 26, 39, 41, 5, 7, 20, 33, 43, 4, 14, 27, 29, 42, 3, 13,$ $15, 17, 30, 40, 1, 11)$ \\
$C_1=(0, 10, 20, 30, 43, 7, 9, 22, 35, 45, 47, 11, 13, 23, 36, 38, 40, 4, 17, 19, 32, 42, 6, 8, 21, 34,$ $  44, 5, 18, 31, 33, 46, 48, 12, 14, 24, 26, 28, 41, 2, 15, 25, 27, 37, 1, 3, 16, 29, 39, 0)$ \\
$C_2=(0, 13, 26, 36, 46, 7, 17, 27, 40, 42, 44, 8, 10, 12, 25, 38, 2, 4, 6, 19, 29, 31, 41, 43, 45, 9,$  $11, 21, 23, 33, 35, 48, 1, 14, 16, 18, 20, 22, 24, 37, 39, 3, 5, 15, 28, 30, 32, 34, 47, 0)$ \\
Partition contains:  $\{\pm 7,\pm 1\}$, $\{\pm 14,\pm 3\}$, $\{\pm 21,\pm 4\}$, and $\{ e \}$ for each remaining difference $e$

\end{itemize}

\section{Computational results --- Case $m \equiv 0 \pmod{3}$}

For each value of $m$ we give sets $S^X,S^Y \subseteq \ZZ_m^\ast$ satisfying Conditions ($X_1$) -- ($X_4$), ($Y_1$) -- ($Y_4$) of Proposition~\ref{pro:=0} (if $m \ge 15$), or from the proof of Lemma~\ref{lem:9} (if $m=9$). The required differences appear in bold type. In addition, we give a desired decomposition of a subgraph of $D[X]$ into directed $m$-cycles $C_i'$ and (for $m=9$ only) pairwise vertex-disjoint directed paths $Q_i'$, and
a desired decomposition of a subgraph of $D[Y]$ into directed $m$-cycles $C_i$ and pairwise vertex-disjoint directed paths $Q_i$. We also give a partition of $\ZZ_m^\ast-(S \cup \{ \frac{m+1}{2} \})$ satisfying the assumptions of Lemma~\ref{lem:BerFavMah}.

\begin{itemize}
\item $m=9$ \\
$S^X=\{ 1,2,3,\mathbf{4}, 6,\mathbf{7} \}$ \\
$Q_1'=(1, 3)$ \\
$Q_2'=(2, 4)$ \\
$Q_3'=(0, 6, 7, 8, 5)$ \\
$C_1'=(0, 4, 8, 3, 7, 5, 6, 1, 2, 0)$ \\
$C_2'=(0, 7, 2, 6, 8, 1, 4, 5, 3, 0)$ \\
$C_3'=(0, 3, 6, 4, 7, 1, 5, 2, 8, 0)$ \\
$C_4'=(0, 1, 8, 2, 3, 5, 7, 4, 6, 0)$ \\
$C_5'=(0, 2, 5, 8, 6, 3, 4, 1, 7, 0)$ \\
Partition contains: $\{ 8 \}$

\smallskip

$S^Y=\{ \mathbf{1}, 3,4,6,\mathbf{7},8 \}$ \\
$Q_1=(2, 3)$ \\
$Q_2=(4, 7, 1)$ \\
$Q_3=(6, 5, 0, 8)$ \\
$C_1=(0, 1, 8, 2, 5, 6, 7, 4, 3, 0)$ \\
$C_2=(0, 7, 8, 5, 3, 1, 4, 2, 6, 0)$ \\
$C_3=(0, 3, 6, 4, 1, 7, 5, 2, 8, 0)$ \\
$C_4=(0, 6, 1, 5, 4, 8, 3, 7, 2, 0)$ \\
$C_5=(0, 4, 5, 8, 7, 6, 3, 2, 1, 0)$ \\
Partition contains: $\{ 2 \}$

\item $m=15$ \\
$S^X=\{ \mathbf{4, 7}, 9\}$ \\
$C_1'=(0, 4, 13, 7, 11, 5, 9, 3, 12, 1, 10, 14, 8, 2, 6, 0)$ \\
$C_2'=(0, 7, 1, 8, 12, 6, 13, 5, 14, 3, 10, 4, 11, 2, 9, 0)$ \\
$C_3'=(0, 9, 13, 2, 11, 3, 7, 14, 6, 10, 1, 5, 12, 4, 8, 0)$ \\
Partition contains:  $\{\pm 3,\pm 5\}$,  and $\{ e \}$ for each remaining difference $e$

\smallskip

$S^Y=\{ \mathbf{1}, 5, 6, 9, \mathbf{10} \}$ \\
$Q_1=(11, 6, 7, 12, 2, 8, 9, 4, 5, 14)$ \\
$Q_2=(13, 3)$ \\
$Q_3=(0, 10, 1)$ \\
$C_1=(0, 1, 2, 12, 7, 13, 8, 3, 4, 14, 9, 10, 11, 5, 6, 0)$ \\
$C_2=(0, 5, 11, 2, 7, 1, 6, 12, 3, 9, 14, 8, 13, 4, 10, 0)$ \\
$C_3=(0, 6, 11, 1, 7, 8, 2, 3, 12, 13, 14, 5, 10, 4, 9, 0)$ \\
$C_4=(0, 9, 3, 8, 14, 4, 13, 7, 2, 11, 12, 6, 1, 10, 5, 0)$ \\
Partition contains:  $\{\pm 3,\pm 2\}$,  and $\{ e \}$ for each remaining difference $e$

\item $m=21$ \\
$S^X=\{ \mathbf{1, 4}, 18\}$ \\
$C_1'=(0, 1, 2, 3, 4, 5, 6, 7, 8, 9, 10, 11, 12, 13, 14, 15, 19, 16, 20, 17, 18, 0)$ \\
$C_2'=(0, 4, 8, 12, 9, 6, 10, 14, 18, 19, 1, 5, 2, 20, 3, 7, 11, 15, 16, 13, 17, 0)$ \\
$C_3'=(0, 18, 15, 12, 16, 17, 14, 11, 8, 5, 9, 13, 10, 7, 4, 1, 19, 20, 2, 6, 3, 0)$ \\
Partition contains:  $\{\pm 6,\pm 7\}$,  $\{\pm 9,\pm 2\}$,and $\{ e \}$ for each remaining difference $e$

\smallskip

$S^Y=\{ 3,\mathbf{4}, \mathbf{10}, 13,  18\}$ \\
$Q_1=(5, 15, 19, 2, 12, 16, 13, 17, 0, 4, 8, 18, 10, 20)$ \\
$Q_2=(7, 11, 14, 6, 3)$ \\
$Q_3=(9, 1)$ \\
$C_1=(0, 10, 14, 18, 1, 11, 15, 4, 7, 17, 6, 19, 8, 12, 9, 13, 16, 5, 2, 20, 3, 0)$ \\
$C_2=(0, 13, 2, 15, 12, 4, 1, 14, 11, 3, 6, 10, 7, 20, 17, 9, 19, 16, 8, 5, 18, 0)$ \\
$C_3=(0, 3, 16, 19, 1, 4, 14, 17, 20, 2, 6, 9, 12, 15, 18, 7, 10, 13, 5, 8, 11, 0)$ \\
$C_4=(0, 18, 15, 7, 4, 17, 14, 3, 13, 10, 2, 5, 9, 6, 16, 20, 12, 1, 19, 11, 8, 0)$ \\
Partition contains:  $\{\pm 6,\pm 7\}$,  $\{\pm 9,\pm 2\}$, and $\{ e \}$ for each remaining difference $e$

\item $m=27$ \\
$S^X=\{ 3, \mathbf{22, 25} \}$ \\
$C_1'=(0, 25, 23, 21, 19, 17, 15, 13, 11, 6, 4, 1, 26, 2, 24, 22, 20, 18, 16, 14, 9, 12, 7, 10, 5, 8, 3, 0)$ \\
$C_2'=(0, 22, 25, 20, 15, 18, 13, 16, 19, 14, 17, 12, 10, 8, 11, 9, 4, 7, 2, 5, 3, 6, 1, 23, 26, 21, 24, 0)$ \\
$C_3'=(0, 3, 25, 1, 4, 26, 24, 19, 22, 17, 20, 23, 18, 21, 16, 11, 14, 12, 15, 10, 13, 8, 6, 9, 7, 5, 2, 0)$ \\
Partition contains:  $\{\pm 6,\pm 1\}$, $\{\pm 9,\pm 4\}$, $\{\pm 12,\pm 7\}$, and $\{ e \}$ for each remaining difference $e$

\smallskip

$S^Y=\{ 3,4,\mathbf{19}, 24, \mathbf{25} \}$ \\
$Q_1=(20, 12, 4, 2, 5, 9, 13, 17, 21, 19, 16, 8, 11, 15, 18, 10, 7, 26)$ \\
$Q_2=(22, 14, 6, 25, 23, 0, 3)$ \\
$Q_3=(24, 1)$ \\
$C_1=(0, 19, 11, 3, 1, 26, 18, 16, 14, 12, 15, 7, 4, 23, 20, 24, 22, 25, 17, 9, 6, 10, 2, 21, 13, 5, 8, 0)$ \\
$C_2=(0, 25, 2, 6, 4, 7, 11, 8, 12, 9, 1, 5, 24, 16, 13, 10, 14, 17, 20, 23, 21, 18, 15, 19, 22, 26, 3, 0)$ \\
$C_3=(0, 4, 1, 25, 22, 19, 17, 14, 11, 9, 12, 10, 8, 6, 3, 7, 5, 2, 26, 23, 15, 13, 16, 20, 18, 21, 24, 0)$ \\
$C_4=(0, 24, 21, 25, 1, 4, 8, 5, 3, 6, 9, 7, 10, 13, 11, 14, 18, 22, 20, 17, 15, 12, 16, 19, 23, 26, 2, 0)$ \\
Partition contains:  $\{\pm 6,\pm 1\}$, $\{\pm 9,\pm 5\}$, $\{\pm 12,\pm 7\}$,  and $\{ e \}$ for each remaining difference $e$

\item $m=33$ \\
$S^X=\{ 11, 12, \mathbf{19, 22}\}$ \\
$C_1'=(0, 19, 5, 24, 10, 29, 15, 1, 20, 6, 28, 14, 25, 11, 30, 8, 27, 16, 2, 13, 32, 21, 7, 18, 4, 26, 12,$ $ 23, 9, 31, 17, 3, 22, 0)$ \\
$C_2'=(0, 22, 1, 12, 31, 20, 9, 28, 6, 17, 29, 18, 7, 19, 8, 30, 16, 5, 27, 13, 24, 3, 25, 14, 26, 15, 4,$ $ 23, 2, 21, 10, 32, 11, 0)$ \\
$C_3'=(0, 11, 22, 8, 19, 30, 9, 20, 31, 10, 21, 32, 18, 29, 7, 26, 4, 15, 27, 5, 16, 28, 17, 6, 25, 3, 14,$ $ 2, 24, 13, 1, 23, 12, 0)$ \\
$C_4'=(0, 12, 24, 2, 14, 3, 15, 26, 5, 17, 28, 7, 29, 8, 20, 32, 10, 22, 11, 23, 1, 13, 25, 4, 16, 27, 6,$ $ 18, 30, 19, 31, 9, 21, 0)$ \\
Partition contains:  $\{\pm 3,\pm 1\}$, $\{\pm 6,\pm 2\}$, $\{\pm 9,\pm 4\}$,  $\{\pm 15,\pm 5\}$, and $\{ e \}$ for each remaining difference $e$

\smallskip

$S^Y=\{ 1, \mathbf{7, 13}, 26 \}$ \\
$Q_1=(8, 21, 14, 27, 28, 2, 15, 16, 29, 9, 22, 23, 24, 17, 30, 4, 11, 18, 25, 5, 31, 32)$ \\
$Q_2=(12, 19, 20, 13, 26, 6, 7, 0, 1)$ \\
$Q_3=(10, 3)$ \\
$C_1=(0, 7, 20, 27, 1, 14, 21, 28, 8, 15, 22, 29, 30, 23, 16, 9, 2, 3, 4, 17, 10, 11, 24, 31, 5, 12, 13,$ $ 6, 32, 25, 18, 19, 26, 0)$ \\
$C_2=(0, 13, 14, 7, 8, 1, 2, 9, 10, 17, 18, 11, 12, 25, 26, 27, 20, 21, 22, 15, 28, 29, 3, 16, 23, 30,$ $ 31, 24, 4, 5, 6, 19, 32, 0)$ \\
$C_3=(0, 26, 19, 12, 5, 18, 31, 11, 4, 30, 10, 23, 3, 29, 22, 2, 28, 21, 1, 27, 7, 14, 15, 8, 9, 16, 17,$ $ 24, 25, 32, 6, 13, 20, 0)$ \\
Partition contains:  $\{\pm 3,\pm 11\}$, $\{\pm 6,\pm 2\}$, $\{\pm 9,\pm 4\}$, $\{\pm 12,\pm 5\}$, $\{\pm 15,\pm 8\}$,   and $\{ e \}$ for each remaining difference $e$

\item $m=39$ \\
$S^X=\{ \mathbf{13, 16}, 24, 26\}$ \\
$C_1'=(0, 13, 26, 3, 16, 29, 6, 19, 32, 9, 22, 35, 12, 25, 38, 15, 28, 2, 18, 31, 5, 21, 34, 8, 24, 37,$ $ 11, 27, 1, 14, 30, 4, 17, 33, 7, 20, 36, 10, 23, 0)$ \\
$C_2'=(0, 16, 32, 6, 22, 38, 12, 28, 15, 2, 26, 13, 29, 3, 19, 35, 9, 25, 1, 17, 30, 7, 23, 10, 36, 21,$ $ 37, 14, 27, 4, 20, 33, 18, 5, 31, 8, 34, 11, 24, 0)$ \\
$C_3'=(0, 26, 11, 37, 24, 9, 35, 22, 7, 33, 20, 5, 18, 3, 29, 14, 38, 25, 10, 34, 19, 4, 30, 15, 31, 16,$ $ 1, 27, 12, 36, 23, 8, 21, 6, 32, 17, 2, 28, 13, 0)$ \\
$C_4'=(0, 24, 11, 35, 20, 7, 31, 18, 34, 21, 8, 32, 19, 6, 30, 17, 4, 28, 5, 29, 16, 3, 27, 14, 1, 25, 12,$ $ 38, 23, 36, 13, 37, 22, 9, 33, 10, 26, 2, 15, 0)$ \\
Partition contains:  $\{\pm 3,\pm 1\}$, $\{\pm 6,\pm 2\}$, $\{\pm 9,\pm 4\}$, $\{\pm 12,\pm 5\}$, $\{\pm 18,\pm 7\}$,   and $\{ e \}$ for each remaining difference $e$

\smallskip

$S^Y=\{ 2,7, \mathbf{28, 34} \}$ \\
$Q_1=(29, 18, 7, 14, 16, 23, 25, 27, 22, 17, 12, 19, 21, 10, 5, 0, 28, 35, 24, 13, 2, 30, 32, 34, 36,$ $38)$ \\
$Q_2=(31, 20, 9, 37, 26, 15, 4, 11, 6, 8, 3)$ \\
$Q_3=(33, 1)$ \\
$C_1=(0, 34, 23, 12, 1, 35, 3, 37, 5, 7, 2, 36, 4, 32, 21, 16, 18, 25, 20, 15, 17, 6, 13, 8, 10, 38, 27,$ $ 29, 31, 33, 22, 24, 26, 28, 30, 19, 14, 9, 11, 0)$ \\
$C_2=(0, 7, 9, 16, 11, 13, 15, 10, 17, 19, 8, 36, 25, 32, 27, 34, 2, 4, 38, 6, 1, 3, 5, 12, 14, 21, 28,$ $23, 18, 20, 22, 29, 24, 31, 26, 33, 35, 30, 37, 0)$ \\
$C_3=(0, 2, 9, 4, 6, 34, 29, 36, 31, 38, 1, 8, 15, 22, 11, 18, 13, 20, 27, 16, 5, 33, 28, 17, 24, 19, 26,$ $ 21, 23, 30, 25, 14, 3, 10, 12, 7, 35, 37, 32, 0)$ \\
Partition contains:  $\{\pm 3,\pm 13\}$, $\{\pm 6,\pm 1\}$, $\{\pm 9,\pm 4\}$, $\{\pm 12,\pm 8\}$, $\{\pm 15,\pm 10\}$, $\{\pm 18,\pm 14\}$,   and $\{ e \}$ for each remaining difference $e$

\item $m=45$ \\
$S^X=\{ \mathbf{4, 7}, 39\}$ \\
$C_1'=(0, 4, 8, 12, 16, 20, 24, 28, 32, 36, 43, 5, 9, 13, 17, 21, 25, 29, 33, 37, 41, 2, 6, 10, 14, 18,$ $ 22, 26, 30, 34, 38, 42, 1, 40, 44, 3, 7, 11, 15, 19, 23, 27, 31, 35, 39, 0)$ \\
$C_2'=(0, 7, 1, 5, 12, 19, 26, 33, 27, 21, 28, 35, 42, 4, 11, 18, 25, 32, 39, 43, 2, 9, 16, 23, 30, 37,$  $44, 6, 13, 20, 14, 8, 15, 22, 29, 36, 40, 34, 41, 3, 10, 17, 24, 31, 38, 0)$ \\
$C_3'=(0, 39, 33, 40, 1, 8, 2, 41, 35, 29, 23, 17, 11, 5, 44, 38, 32, 26, 20, 27, 34, 28, 22, 16, 10, 4,$ $ 43, 37, 31, 25, 19, 13, 7, 14, 21, 15, 9, 3, 42, 36, 30, 24, 18, 12, 6, 0)$ \\
Partition contains:  $\{\pm 3,\pm 5\}$, $\{\pm 9,\pm 10\}$, $\{\pm 12,\pm 20\}$, $\{\pm 15,\pm 1\}$, $\{\pm 18,\pm 2\}$, $\{\pm 21,\pm 8\}$,   and $\{ e \}$ for each remaining difference $e$

\smallskip

$S^Y=\{ \mathbf{10, 16}, 31, 35\}$ \\
$Q_1=(11, 21, 31, 2, 12, 22, 38, 9, 19, 29, 39, 4, 14, 24, 40, 30, 20, 10, 0, 35, 25, 41, 6, 37, 27, 17,$ $ 7, 42, 28, 44)$ \\
$Q_2=(13, 23, 33, 43, 8, 18, 34, 5, 36, 26, 16, 32, 3)$ \\
$Q_3=(15, 1)$ \\
$C_1=(0, 10, 20, 30, 40, 5, 21, 37, 23, 13, 3, 38, 28, 18, 4, 39, 29, 15, 31, 41, 12, 2, 33, 19, 35, 6,$ $ 16, 26, 36, 1, 32, 22, 8, 24, 34, 44, 9, 25, 11, 42, 7, 17, 27, 43, 14, 0)$ \\
$C_2=(0, 16, 6, 22, 32, 42, 13, 29, 19, 9, 44, 30, 1, 17, 3, 34, 20, 36, 7, 38, 24, 10, 41, 31, 21, 11,$ $ 27, 37, 2, 18, 8, 43, 33, 23, 39, 25, 15, 5, 40, 26, 12, 28, 14, 4, 35, 0)$ \\
$C_3=(0, 31, 17, 33, 4, 20, 6, 41, 27, 13, 44, 34, 24, 14, 30, 16, 2, 37, 8, 39, 10, 26, 42, 32, 18, 28,$ $ 38, 3, 19, 5, 15, 25, 35, 21, 7, 23, 9, 40, 11, 1, 36, 22, 12, 43, 29, 0)$ \\
Partition contains:  $\{\pm 3,\pm 5\}$, $\{\pm 6,\pm 20\}$, $\{\pm 9,\pm 1\}$, $\{\pm 12,\pm 2\}$, $\{\pm 15,\pm 4\}$, $\{\pm 18,\pm 7\}$, $\{\pm 21,\pm 8\}$,   and $\{ e \}$ for each remaining difference $e$

\end{itemize}

\end{document}